\documentclass{SCAEOL}
\numberwithin{equation}{section}
\begin{document}

\Year{2013} %
\Month{January}
\Vol{56} %
\No{1} %
\BeginPage{1} %
\EndPage{XX} %
\AuthorMark{First1 L N {\it et al.}} \ReceivedDay{January 26, 2015}
\AcceptedDay{?? ??, ??}

\newtheorem{algorithm1}{Weak Galerkin Algorithm}
\newtheorem{algorithm2}{Hybridized Weak Galerkin (HWG) Algorithm}
\newtheorem{algorithm3}{Variable Reduction Algorithm}

\title{A Hybridized Weak Galerkin Finite Element Scheme for the Stokes Equations}{}


\author[1]{Qilong Zhai}{}
\author[1]{Ran Zhang}{Corresponding author}
\author[2]{Xiaoshen Wang}{}

\address[{\rm1}]{Department of Mathematics, Jilin
University, Changchun, {\rm 130012}, China;}
\address[{\rm2}]{Department of Mathematics,
University of Arkansas at Little Rock, Little Rock,  {\rm AR
72204},United States;} \Emails{diql@mails.jlu.edu.cn,
zhangran@mail.jlu.edu.cn, xxwang@ualr.edu}\maketitle


 {\begin{center}
\parbox{14.5cm}{\begin{abstract}
In this paper a hybridized weak Galerkin (HWG) finite element method
for solving the Stokes equations in the primary velocity-pressure
formulation is introduced. The WG method uses weak functions and
their weak derivatives which are defined as distributions. Weak
functions and weak derivatives can be approximated by piecewise
polynomials with various degrees. Different combination of
polynomial spaces leads to different WG finite element methods,
which makes WG methods highly flexible and efficient in practical
computation. A Lagrange multiplier is introduced to provide a
numerical approximation for certain derivatives of the exact
solution. With this new feature, HWG method can be used to deal with
jumps of the functions and their flux easily. Optimal order error
estimate are established for the corresponding HWG finite element
approximations for both {\color{black}primal variables} and the
Lagrange multiplier. A Schur complement formulation of the HWG
method is derived for implementation purpose. The validity of the
theoretical results is demonstrated in numerical tests.\vspace{-3mm}
\end{abstract}}\end{center}}

 \keywords{hybridized weak Galerkin finite element methods,
weak gradient, weak divergence, Stokes equation}

 \MSC{65N30, 65N15, 65N12, 74N20}

\renewcommand{\baselinestretch}{1.2}
\begin{center} \renewcommand{\arraystretch}{1.5}
{\begin{tabular}{lp{0.8\textwidth}} \hline \scriptsize {\bf
Citation:}\!\!\!\!&\scriptsize Zhai Q, Zhang R, Wang X. A Hybridized
Weak Galerkin Finite Element Scheme for the Stokes Equations. Sci
China Math, 2013, 56, doi: 10.1007/s11425-000-0000-0\vspace{1mm}
\\
\hline
\end{tabular}}\end{center}

\baselineskip 11pt\parindent=10.8pt  \wuhao
\section{Introduction}

Weak Galerkin (WG) refers to a general finite element technique for
partial differential equations (PDEs) in which differential
operators are approximated by their weak forms as distributions.
Since their introduction, WG finite element methods have been
applied successfully to the discretization of several classes of
partial differential equations, e.g., second order elliptic
equations \cite{WY13, WY14, CWY14, MWY14, MWWY13_2, MWWYZ13}, the
Biharmonic equations \cite{MWY, MWWY13, MWYZ14_2, ZZ}, the Stokes
equations \cite{WY3}, and the Brinkman equations \cite{MWY14_2}. WG
method methods, by design, make use of discontinuous piecewise
polynomials on finite element partitions with arbitrary shape of
polygons and polyhedrons. Weak functions and weak derivatives can be
approximated by piecewise polynomials with various degrees. The
flexibility of {\color{black}WG method} on these aspects of
approximating polynomials makes it an excellent candidate for the
numerical solution of incompressible flow problems.

Hybridization of finite element methods is a technique where
Lagrange multipliers are introduced to relax certain constrains such
as some continuity requirements. The main feature of the HWG method
is that their approximate solutions can be expressed in an
element-by-element fashion. Hybridization \cite{AB85} can be
employed to obtain an efficient implementation for solving PDEs. The
generalization of this idea to mixed finite elements has been
investigated in \cite{Babuska, B74, BDM85, BBM92, RT77}. The idea of
hybridization was also used in discontinuous Galerkin methods
\cite{ES10, Lehrenfeld10, NPC10} to derive hybridizable
discontinuous Galerkin (HDG) {\color{black}\cite{CS13, CS14,
CGL74}}.

In this paper, the WG finite element formulation developed in
\cite{WY3} is hybridized to obtain our new hybridized weak Galerkin
finite element method for solving Stokes equations.
{\color{black}This HWG formulation can be modified easily to solve
interface problems by adding two functionals arising from the jump
condition to the right-hand side. A Schur complement formulation of
the HWG method is derived for implementation purpose. By eliminating
the interior unknowns and the Lagrange multipliers, the Schur
complement
formulation yields a system with much smaller {\color{black}size. 
We} shall show that hybridization is a natural approach for the weak
Galerkin finite element method of \cite{WY3}. We shall also
establish a theoretical foundation to address critical issues such
as stability and convergence for the HWG finite element method.}

The paper is organized as follows. In Section 2, we briefly discuss
the continuous Stokes problem and recall some basic results for
later reference. After presenting some standard notations in Sobolev
spaces in Section 3, we introduce two weakly-defined differential
operators: weak gradient and weak divergence. The HWG finite element
scheme for the Stokes problem is developed in Section 4. In Section
5, we shall study the stability and solvability of the HWG scheme.
In particular, the usual inf- sup condition is established for the
HWG scheme. In Section 6, we shall derive an error equation for the
HWG approximations. Optimal-order error estimates for the WG finite
element approximations are also derived in this Section. The
equivalence of HWG formulation and its Schur complement formulation
is proved in Section 7. Finally in Section 8, numerical experiments
are conducted.

\section{The model problem}

Let $\Omega\subset \mathbb{R}^d$ be a polygonal or polyhedral domain
in $\mathbb{R}^d$ for $d=2, 3$ respectively. As a model for the flow
of an incompressible viscous fluid confined in $\Omega$, we consider
the stationary Stokes problem with nonhomogeneous Dirichlet boudary
conditions, given by
\begin{eqnarray}\label{Stokes}
-\Delta \textbf{u}+\nabla p &=& \textbf{f}, \quad {\rm in}\ \Omega,\\
\label{divEquation}\nabla \cdot\textbf{u}&=&0,\quad {\rm in}\
\Omega,
\\
\label{boundarycond}\textbf{u}&=&\textbf{g},\quad {\rm on}\
\partial\Omega.
\end{eqnarray}
Throughout the presentation, we assume that the unit external
volumetric force acting on the fluid $\textbf{f}\in
[L^2(\Omega)]^d$.

The weak form in the primary velocity-pressure formulation for the
Stokes problem (\ref{Stokes})-(\ref{boundarycond}) seeks
$\textbf{u}\in [H^1(\Omega)]^d$ and $p\in L^2_0(\Omega)$ satisfying
$\textbf{u}=\textbf{g}$ on $\partial \Omega$ and
\begin{eqnarray}\label{variational_Stokes}
(\nabla\textbf{u}, \nabla \mathbf{v})-(\nabla \cdot \mathbf{v}, p)
&=& (\textbf{f},\mathbf{v}),
\\
\label{variational_Div}(\nabla \cdot\textbf{u}, q)&=&0,
\end{eqnarray}
for all $\mathbf{v}\in [H^1_0(\Omega)]^d$ and $q\in L^2_0(\Omega)$.

Recently a weak Galerkin finite element method has been developed
for solving the Stokes equations in \cite{WY3}. The main idea of
weak Galerkin finite element methods is the introduction of weak
functions and their corresponding weak derivatives in the algorithm
design. With well defined weak functions and weak derivatives, a
weak Galerkin finite element formulation for the Stokes equations is
derived from the variational form of the PDE
(\ref{variational_Stokes})-(\ref{variational_Div}) by replacing
regular derivatives with weak derivatives and possibly adding a
parameter independent stabilizer: find $\textbf{u}_h$ and $p_h$ from
properly-defined finite element spaces satisfying
\begin{eqnarray}\label{weak_Stokes}
(\nabla_w\textbf{u}_h, \nabla_w \mathbf{v})-(\nabla_w \cdot
\mathbf{v}, p_h) +s(\textbf{u}_h,\mathbf{v})&=&
(\textbf{f},\mathbf{v}),
\\
\label{weak_Div}(\nabla_w\cdot\textbf{u}_h, q)&=&0
\end{eqnarray}
for all test functions $\mathbf{v}$ and $q$ in test spaces. In this
paper, the WG finite element formulation developed in \cite{WY3} is
hybridized to obtain our new hybridized weak Galerkin finite element
method for solving Stokes equation
(\ref{Stokes})-(\ref{boundarycond}).

\section{Weak Differential Operators and Discrete Weak Gradient}

Let $D$ be any open bounded domain with Lipschitz continuous
boundary in $\mathbb{R}^d, d=2, 3$. We use the standard definition
for the Sobolev space $H^s(D)$ and their associated inner products
$(\cdot, \cdot)_{s, D}$, norms $\|\cdot\|_{s, D}$, and seminorms
$|\cdot|_{s, D}$ for any $s\ge 0$. For example, for any integer
$s\ge 0$, the seminorm $|\cdot|_{s, D}$ is given by
$$
|v|_{s, D}=\left(\sum_{|\alpha|=s}\int_D |\partial^\alpha v|^2 {\rm
d} D\right)^{\frac{1}{2}}
$$
with the usual notation
$$
\alpha=(\alpha_1, \cdots, \alpha_d), \quad
|\alpha|=\alpha_1+\cdots+\alpha_d, \quad
\partial^\alpha=\prod_{j=1}^d \partial_{x_j}^{\alpha_j}.
$$
The Sobolev norm $\|\cdot\|_{m,D}$ is given by
$$
\|v\|_{m, D}=\left(\sum_{j=0}^{m} |v|_{j, D}^2\right)^{\frac{1}{2}}.
$$

The space $H^0(D)$ coincides with $L^2(D)$, for which the norm and
the inner product are denoted by $\|\cdot\|_D$ and
$(\cdot,\cdot)_D$, respectively. When $D=\Omega$, we shall drop the
subscript $D$ in the norm and in the inner product notation.

The space $H({\rm div}; D)$  is defined as the set of vector-valued
functions on $D$ which, together with their divergence, are square
integrable; i.e.,
$$
H({\rm div}; D)=\{\mathbf{v}: \mathbf{v}\in [L^2(D)]^d, \nabla \cdot
\mathbf{v}\in L^2(D)\}.
$$

Let $T$ be a polygonal or polyhedral domain with boundary $\partial
T$. A weak vector-valued function on the region $T$ refers to a
vector-valued function $\mathbf{v}=\{\mathbf{v}_0, \mathbf{v}_b\}$
such that $\mathbf{v}_0\in [L^2(T)]^d$ and $\mathbf{v}_b\in
[H^{\frac12}(\partial T)]^d$. Let
\begin{eqnarray}\label{GradSpace_Weakfunction}
\mathcal{V}(T)=\{\mathbf{v}=\{\mathbf{v}_0,
\mathbf{v}_b\}:¡¡\mathbf{v}_0\in [L^2(T)]^d, \mathbf{v}_b\in
[H^{\frac12}(\partial T)]^d\}
\end{eqnarray}

Recall that, for any $\mathbf{v}\in \mathcal{V}(T)$, the {\rm weak
gradient} of $\mathbf{v}$ is defined as a linear functional
$\nabla_w \mathbf{v}$ in the dual space of $[H(div,T)]^d$ whose
action on each $q\in [H(div,T)]^d$ is given by
\begin{eqnarray}
\label{wGradient} (\nabla_w \mathbf{v}, q)_T=-(\mathbf{v}_0,
\nabla\cdot q)_T+\langle \mathbf{v}_b, q\cdot
\textbf{n}\rangle_{\partial T},
\end{eqnarray}
where $\textbf{n}$ is the outer unit normal vector along $\partial
T$, $(\cdot, \cdot)_T$ stands for the $L^2$-inner product in
$[L^2(T)]^d$ and $\langle\cdot, \cdot\rangle_{\partial T}$ is the
inner product in $[H^{\frac12}(\partial T)]^d$.

A discrete version of the weak gradient operator $\nabla_w$ is an
approximation, denoted by $\nabla_{w,r,T}$ in the space of
polynomials of degree $r$ such that
\begin{eqnarray}
\label{Discrete_wGradient} (\nabla_{w,r,T} \mathbf{v},
q)_T=-(\mathbf{v}_0, \nabla \cdot q)_T+\langle \mathbf{v}_b, q\cdot
\textbf{n}\rangle_{\partial T},\quad \forall q\in [P_r(T)]^{d\times
d}.
\end{eqnarray}

From the integration by parts, we have
$$
(\mathbf{v}_0, \nabla q)_T=-(\nabla \mathbf{v}_0, q)_T+\langle
\mathbf{v}_0, q\cdot \textbf{n}\rangle_{\partial T},
$$
Substituting the above identity into (\ref{Discrete_wGradient})
yields
\begin{eqnarray}
\label{Discrete_wGradient-useful} \qquad(\nabla_{w,r,T} \mathbf{v},
q)_T-(\nabla \mathbf{v}_0, q)_T=\langle \mathbf{v}_b-\mathbf{v}_0,
q\cdot \textbf{n}\rangle_{\partial T}
\end{eqnarray}
for all $q\in [P_r(T)]^{d\times d}$.

To define a weak divergence, we require weak function
$\mathbf{v}=\{\mathbf{v}_0, \mathbf{v}_b\}$ to be such that
$\mathbf{v}_0\in [L^2(T)]^d$ and $\mathbf{v}_b\cdot\textbf{n}\in
H^{-\frac12}(\partial T)$. Denote
\begin{eqnarray}\label{DivSpace_Weakfunction}
V(T)=\{\mathbf{v}=\{\mathbf{v}_0, \mathbf{v}_b\}:¡¡\mathbf{v}_0\in
[L^2(T)]^d, \mathbf{v}_b\cdot\textbf{n}\in H^{-\frac12}(\partial
T)\}
\end{eqnarray}

Recall that, for any $\mathbf{v}\in \mathcal{V}(T)$, the {\rm weak
divergence} of $\mathbf{v}$ is defined as a linear functional
$\nabla_w\cdot \mathbf{v}$ in the dual space of $H^1(T)$ whose
action on each $\varphi\in H^1(T)$ is given by
\begin{eqnarray}
\label{wDivergence} (\nabla_w \cdot\mathbf{v},
\varphi)_T=-(\mathbf{v}_0, \nabla \varphi)_T+\langle
\mathbf{v}_b\cdot \textbf{n}, \varphi\rangle_{\partial T},
\end{eqnarray}
where $\textbf{n}$ is the outer unit normal vector along $\partial
T$, $(\cdot, \cdot)_T$ stands for the $L^2$-inner product in
$L^2(T)$ and $\langle\cdot, \cdot\rangle_{\partial T}$ is the inner
product in $H^{\frac12}(\partial T)$.

A discrete version of the weak divergence operator $\nabla_{w}\cdot$
is an approximation, denoted by $(\nabla_{w,r,T}\cdot)$ in the space
of polynomials of degree $r$ such that
\begin{eqnarray}
\label{Discrete_wDivergence} (\nabla_{w,r,T}\cdot \mathbf{v},
\varphi)_T=-(\mathbf{v}_0, \nabla \varphi)_T+\langle
\mathbf{v}_b\cdot \textbf{n}, \varphi\rangle_{\partial T},\quad
\forall \varphi\in P_r(T).
\end{eqnarray}

From the integration by parts, we have
$$
(\mathbf{v}_0, \nabla \varphi)_T=-(\nabla \cdot \mathbf{v}_0,
\varphi)_T+\langle \mathbf{v}_0\cdot \textbf{n},
\varphi\rangle_{\partial T}.
$$
Substituting the above identity into (\ref{Discrete_wGradient})
yields
\begin{eqnarray}
\label{Discrete_wDivergence-useful} \qquad(\nabla_{w,r,T}\cdot
\mathbf{v}, \varphi)_T-(\nabla \cdot\mathbf{v}_0, \varphi)_T=\langle
(\mathbf{v}_b-\mathbf{v}_0)\cdot \textbf{n},
\varphi\rangle_{\partial T}
\end{eqnarray}
for all $\varphi\in P_r(T)$.

\section{A Hybridized Weak Galerkin Formulation}

The goal of this section is to introduce a hybridized formulation
for the weak Galerkin finite element algorithm that was first
designed in \cite{WY3}.

\subsection{Notations}

Let $\mathcal{T}_h$ be a partition of the domain $\Omega$ into
polygons in 2D or polyhedra in 3D. Assume that $\mathcal{T}_h$ is
shape regular in the sense as defined in \cite{WY14}. Denote by
$\mathcal{E}_h$ the set of all edges or flat faces in
$\mathcal{T}_h$, and let
$\mathcal{E}_h^0=\mathcal{E}_h\setminus\partial \Omega$ be the set
of all interior edges or flat faces. Denote by $h_T$ the diameter of
$T\in \mathcal{T}_h$ and $h=\max_{T\in \mathcal{T}_h}h_T$ the
meshsize for the partition $\mathcal{T}_h$.

On each element $T\in \mathcal{T}_h$, there are spaces of weak
function $\mathcal{V}(T)$ and $V(T)$ defined as in
(\ref{GradSpace_Weakfunction}) and (\ref{DivSpace_Weakfunction}),
respectively. Denote by $\mathcal{V}$ and $\Lambda$ the function
space on $\mathcal{T}_h$ and $\mathcal{E}_h$ given respectively by
\begin{eqnarray}
\label{FunctionalSpace} \mathcal{V}=\prod_{T\in
\mathcal{T}_h}\mathcal{V}(T),\quad \Lambda=\prod_{T\in
\mathcal{T}_h}[H^{\frac12}(\partial T)]^d.
\end{eqnarray}
Note that the values of functions in the spaces $\mathcal{V}(T_1)$
and $\mathcal{V}(T_2)$ are not related for any elements $T_1$ and
$T_2$, even if $T_1$ and $T_2$ share an interior edge or flat face
$e\in \mathcal{E}_h^0$. The jump of $\mathbf{v}=\{\mathbf{v}_0,
\mathbf{v}_b\}$ on $e$ is given by
\begin{eqnarray}
\label{JumpFunction} [\![\mathbf{v}]\!]_e=\left\{\begin{array}{ll}
\mathbf{v}_b|_{\partial T_1}-\mathbf{v}_b|_{\partial T_2}, \quad &
e\in \mathcal{E}_h^0,
\\
\mathbf{v}_b, & e\in\partial\Omega,
\end{array}\right.
\end{eqnarray}
where $\mathbf{v}_b|_{\partial T_i}$ is the value of $\mathbf{v}$ on
$e$ as seen from the element $T_i$. The order of $T_1$ and $T_2$ is
non-essential as long as the difference is taken in a consistent way
in all the formulas. Analogously, for any function
$\boldsymbol{\lambda}\in \Lambda$, we define its similarity on $e\in
\mathcal{E}_h$ by
\begin{eqnarray}
\label{Similarity} \langle\!\langle
\boldsymbol{\lambda}\rangle\!\rangle_e=\left\{\begin{array}{ll}
\boldsymbol{\lambda}|_{\partial T_1}+\boldsymbol{\lambda}|_{\partial
T_2}, \quad & e\in \mathcal{E}_h^0,
\\
0, & e\in\partial\Omega.
\end{array}\right.
\end{eqnarray}
Denote by $\langle\!\langle \boldsymbol{\lambda}\rangle\!\rangle$
the similarity of $\boldsymbol{\lambda}$ in $\mathcal{E}_h$.

For any integer $k\ge 1$, denote by $W_k(T)$ the discrete function
space as follows:
$$
W_k(T)=\{q: q\in L_0^2(\Omega), q|_T\in P_{k-1}(T)\}.
$$
Let $V_k(T)$ denote the discrete weak function space as follows:
$$
V_k(T)=\{\mathbf{v}=\{\mathbf{v}_0, \mathbf{v}_b\}:¡¡\{\mathbf{v}_0,
\mathbf{v}_b\}|_T \in [P_k(T)]^d\times [P_{k-1}(e)]^d, e\subset
\partial T\}.
$$
Let $\Lambda_{k}(\partial T)$  denote
$$
\Lambda_k(\partial T)=\{\boldsymbol{\lambda}:
\boldsymbol{\lambda}|_e\in [P_{k-1}(e)]^d, e\subset
\partial T\}.
$$
By patching $W_k(T)$, $V_k(T)$, and $\Lambda_{k}(\partial T)$ over
all the elements $T\in \mathcal{T}_h$, we obtain three weak Galerkin
finite element spaces $W_h$, $V_h$, and $\Lambda_h$ given by
\begin{eqnarray}
\label{WeakFunctionalSpace} W_h=\prod_{T\in \mathcal{T}_h}W_k(T),
\quad V_h=\prod_{T\in \mathcal{T}_h}V_k(T),\quad
\Lambda_h=\prod_{T\in \mathcal{T}_h}\Lambda_k(\partial T).
\end{eqnarray}

Denote by $V_h^0$ the subspace of $V_h$ consisting of discrete weak
functions with vanishing boundary value
$$
V_h^0=\{\mathbf{v}=\{\mathbf{v}_0, \mathbf{v}_b\}\in
V_h:¡¡\mathbf{v}_b=0\ \  {\rm on} \ \partial \Omega\}.
$$
Furthermore, let $\mathcal{V}_h$ be the subspace of $V_h$ consisting
of functions without jump on each interior edge or flat face
$$
\mathcal{V}_h=\{\mathbf{v}\in V_h: [\![\mathbf{v}]\!]_e=0, e\in
\mathcal{E}_h^0\}.
$$
Denote by $\mathcal{V}_h^0$ a subspace of $\mathcal{V}_h$ consisting
of functions with vanishing boundary values
$$
\mathcal{V}_h^0=\{\mathbf{v}\in \mathcal{V}_h: \mathbf{v}_b|_e=0,
e\in
\partial \Omega\}.
$$
Let $\Xi_h$ be the subspace of {\color{black}$\Lambda_h$} consisting
of functions with similarity zero across each edge or flat face
$$
\Xi_h=\{\boldsymbol{\lambda}\in \Lambda_h: \langle\!\langle
\boldsymbol{\lambda}\rangle\!\rangle_e =0,e\in \mathcal{E}_h\}.
$$
The functions in the space $\Xi_h$ serve as Lagrange multipliers in
hybridization methods.

Denote by $\nabla_{w,k-1}$ and $(\nabla_{w,k-1}\cdot)$ the discrete
weak gradient and the discrete weak divergence on the finite element
space $V_h$. They can be computed by using
(\ref{Discrete_wGradient}) and (\ref{Discrete_wDivergence}) on each
element $T$, respectively.

For each element $T\in \mathcal{T}_h$, denote by $Q_0$ the $L^2$
projection operator from $[L^2(T)]^d$ onto $[P_k(T)]^d$. For each
edge or face $e\in \mathcal{E}_h$, denote by $Q_b$ the $L^2$
projection from $[L^2(e)]^d$ onto $[P_{k-1}(e)]^d$. We shall combine
$Q_0$ with $Q_b$ by writing $Q_h=\{Q_0, Q_b\}$.

\subsection{Algorithm}
On each element $T\in \mathcal{T}_h$, we introduce four bilinear
forms given below:
\begin{eqnarray}
\label{bilinear_sT}s_T(\mathbf{v}, \mathbf{w})&=& h_T^{-1}\langle
Q_b \mathbf{v}_0-\mathbf{v}_b, Q_b
\mathbf{w}_0-\mathbf{w}_b\rangle_{\partial T},
\\
\label{bilinear_aT}a_T(\mathbf{v}, \mathbf{w})&=& (\nabla_w
\mathbf{v}, \nabla_w \mathbf{w})_T+s_T(\mathbf{v}, \mathbf{w}),
\\
\label{bilinear_bT}b_{T}(\mathbf{v}, q)&=&(\nabla_w \cdot
\mathbf{v}, q)_T,
\\
\label{bilinear_cT}c_{T}(\mathbf{v},
\boldsymbol{\lambda})&=&\langle\mathbf{v}_b,
\boldsymbol{\lambda}\rangle_{\partial T},
\end{eqnarray}
for $\mathbf{v}=\{\mathbf{v}_0, \mathbf{v}_b\}\in V_k(T)$,
$\textbf{w}=\{\textbf{w}_0, \textbf{w}_b\}\in V_k(T)$, $q\in W_h(T)$
and $\boldsymbol{\lambda} \in \Lambda_{k}(\partial T)$.

Their sums over all $T\in \mathcal{T}_h$ {\color{black}yield} four
globally-defined bilinear forms:
\begin{eqnarray}
\label{bilinear_s}s(\mathbf{v}, \mathbf{w})&=&
\sum_{T\in\mathcal{T}_h}s_T(\mathbf{v}, \mathbf{w}), \quad
\mathbf{v},\mathbf{w}\in V_h,
\\
\label{bilinear_a}a(\mathbf{v}, \mathbf{w})&=&
\sum_{T\in\mathcal{T}_h}a_{T}(\mathbf{v}, \mathbf{w}),\quad
\mathbf{v},\mathbf{w}\in V_h,
\\
\label{bilinear_b}b(\mathbf{v}, q)&=&
\sum_{T\in\mathcal{T}_h}b_{T}(\mathbf{v}, q),\quad \mathbf{v}\in
V_h,q\in W_h,
\\
\label{bilinear_c}c(\mathbf{v}, \boldsymbol{\lambda})&=&
\sum_{T\in\mathcal{T}_h}c_{T}(\mathbf{v},
\boldsymbol{\lambda}),\quad \mathbf{v}\in
V_h,\boldsymbol{\lambda}\in \Lambda_h.
\end{eqnarray}

The following weak Galerkin finite element scheme for the Stokes
equation (\ref{Stokes}) was introduced and analyzed in \cite{WY3}.
\begin{algorithm1}\label{WGA}
A numerical approximation for (\ref{Stokes})-(\ref{boundarycond})
can be obtained by seeking $\bar{\mathbf{u}}_h=\{\bar{\mathbf{u}}_0,
\bar{\mathbf{u}}_b\}\in V_h$ and $\bar{p}_h\in W_h$ such that
$\bar{\mathbf{u}}_b=Q_b \textbf{g}$ on $\partial \Omega$ and
\begin{eqnarray}\label{WG1}
a(\bar{\mathbf{u}}_h, \mathbf{v})-b(\mathbf{v},
\bar{p}_h)&=&(\mathbf{f}, \mathbf{v}_0),
\\
\label{WG2}b(\bar{\mathbf{u}}_h, q)&=& 0,
\end{eqnarray}
for all $\mathbf{v}=\{\mathbf{v}_0, \mathbf{v}_b\}\in
\mathcal{V}_h^0$ and $q\in W_h$.
\end{algorithm1}

The weak Galerkin finite element algorithm 1 can be hybridized in
the finite element space $\mathcal{V}_h$ by using a Lagrange
multiplier that shall enforce the continuity of the functions in
$V_h$ on interior element boundaries. The corresponding formulation
can be described as follows.

\begin{algorithm2}\label{HWGA}
A numerical approximation for (\ref{Stokes})-(\ref{boundarycond})
can be obtained by seeking
$(\mathbf{u}_h;p_h;\boldsymbol{\lambda}_h)\in V_h\times W_h\times
\Xi_h$ such that $\mathbf{u}_b=Q_b \textbf{g}$ on $\partial \Omega$
and
\begin{eqnarray}\label{HWG1}
a(\mathbf{u}_h, \mathbf{v})-b(\mathbf{v},
p_h)-c(\mathbf{v},\boldsymbol{\lambda}_h)&=&(\mathbf{f},
\mathbf{v}_0),
\\
\label{HWG2}b(\mathbf{u}_h, q)+c(\mathbf{u}_h,\boldsymbol{\mu})&=&
0,
\end{eqnarray}
for all $\mathbf{v}=\{\mathbf{v}_0, \mathbf{v}_b\}\in V^0_h$, $q\in
W_h$ and $\boldsymbol{\mu}\in\Xi_h$.
\end{algorithm2}


\begin{lemma}
The WG finite element scheme (\ref{HWG1})-(\ref{HWG2}) has a unique
solution.
\end{lemma}
\begin{proof}
Let $\mathbf{f}=\textbf{0}$, we shall show that the solution of
(\ref{HWG1})-(\ref{HWG2}) is trivial. To this end, taking
$\mathbf{v}=\mathbf{u}_h$, $q=p_h$, and
$\boldsymbol{\mu}=\boldsymbol{\lambda}_h$ and subtracting
(\ref{HWG2}) from (\ref{HWG1}) we arrive at
$$
a(\mathbf{u}_h,\mathbf{u}_h)=0.
$$
By the definition of $a(\cdot,\cdot)$, we know
$\nabla_w\mathbf{u}_h=0$ on each $T\in\mathcal{T}_h$, $
\mathbf{u}_0=\mathbf{u}_b$ on each $\partial T$.

By (\ref{Discrete_wDivergence-useful}) and the fact that
$\mathbf{u}_b=\mathbf{u}_0$ on $\partial T$ we have, for any
$\tau\in[P_{k-1}(T)]^{d\times d}$,
\begin{eqnarray*}
0&=&(\nabla_w\mathbf{u}_h,\tau)_T
\\
&=&(\nabla\mathbf{u}_0,\tau)_T-\langle\mathbf{u}_0-\mathbf{u}_b,\tau\cdot\mathbf{n}\rangle_{\partial
T}
\\
&=&(\nabla\mathbf{u}_0,\tau)_T,
\end{eqnarray*}
which implies $\nabla\mathbf{u}_0=0$ on each $T\in\mathcal{T}_h$ and
thus $\mathbf{u}_0$ is a constant. Since $\mathbf{u}_0=\mathbf{u}_b$
on each $\partial T$, we have
\begin{eqnarray*}
b(\mathbf{u}_h,q)=-(\mathbf{u}_0, \nabla q)+
\langle\mathbf{u}_0\cdot \mathbf{n}, q\rangle=(\nabla \cdot
\mathbf{u}_h , q)=0.
\end{eqnarray*}
From (\ref{HWG2}), we obtain
\begin{eqnarray*}
c(\mathbf{u}_h,\boldsymbol{\mu})=0.
\end{eqnarray*}
Let $\boldsymbol{\mu}=[\![\mathbf{u}_b]\!]$ in above equation, then
it follows that
$$
\sum_{T\in\mathcal{T}_h} h\|[\![\mathbf{u}_b]\!]\|_e^2=0.
$$
Thus $[\![\mathbf{u}_b]\!]=0$, which implies that $\mathbf{u}_0$ is
continuous and we arrive at $\mathbf{u}_h=\{\b0, \b0\}$ in $\Omega$.

Let $\mathbf{v}_b=0$ in (\ref{HWG1}), then it follows from
$\mathbf{u}_h=\{\b0,\b0\}$ and $\mathbf{f}=\b0$ that
\begin{eqnarray*}
b(\mathbf{v},p_h)=(\nabla_w \cdot \mathbf{v}, p_h)=-(\mathbf{v}_0,
\nabla p_h)=0.
\end{eqnarray*}
Hence we have $\nabla p_h=0$ on each $T\in\mathcal{T}_h$. Thus $p_h$
is a constant in $T$. Let $\mathbf{v}_b|_e=[\![p_h]\!]_e,
\mathbf{v}_0=0$. Then
$$
c(\mathbf{v}, \boldsymbol{\lambda}_h)=0.
$$
Thus
$$
0=b(\mathbf{v}, p_h)=\sum_{e\in\mathcal{E}_h} \|[\![p_h]\!]\|_e^2.
$$
Hence $p_h$ is continuous. From $p_h\in L_0^2(\Omega)$, we would
obtain $p_h=0$ in $\Omega$.

Finally, Let $\mathbf{v}_b=\boldsymbol{\lambda}_h$, from
$\mathbf{u}_h=\{\b0,\b0\}$ and $p_h=0$ in $\Omega$, we obtain
$$
c(\mathbf{v}, \boldsymbol{\lambda}_h)=0,
$$
which means $\boldsymbol{\lambda}_h=0$.

This completes the proof of the lemma.
\end{proof}

\subsection{The Relation between WG and HWG}

The rest of this section will show that the above two schemes are
equivalent in that the solutions $\bar{\mathbf{u}}_h, \bar{p}_h$
from (\ref{WG1})-(\ref{WG2}) and $\mathbf{u}_h, p_h$ from
(\ref{HWG1})-(\ref{HWG2}) are identical, respectively.

For any $\mathbf{v}\in \mathcal{V}^0_h$, let
\begin{eqnarray}\label{tripbarnorm}
{|\hspace{-.02in}|\hspace{-.02in}|} \mathbf{v}
{|\hspace{-.02in}|\hspace{-.02in}|}^2=a(\mathbf{v},\mathbf{v})=\sum_{T\in\mathcal{T}_h}\|\nabla_w\mathbf{v}\|_T^2+\sum_{T\in\mathcal{T}_h}h_T^{-1}\|Q_b\mathbf{v}_0-\mathbf{v}_b\|_{\partial
T}^2.
\end{eqnarray}
It has been verified in \cite{WY3} that (\ref{tripbarnorm}) defines
a norm in the vector space $\mathcal{V}^0_h$.
\begin{theorem}\label{theorem:wg-hwg}
Let $\mathbf{u}_h\in V_h$, $p_h\in W_h$ be the first two components
of the solution of the hybridized WG algorithm
(\ref{HWG1})-(\ref{HWG2}). Then, we have $[\![\mathbf{u}]\!]_e=0$
for all $e\in\mathcal{E}_h^0$; i.e., $\mathbf{u}\in \mathcal{V}_h$
and $\mathbf{u}_b=Q_b \textbf{g}$ on $\partial \Omega$. Furthermore,
we have that $\mathbf{u}_h$ and $p_h$ satisfy the equation
(\ref{WG1})-(\ref{WG2}), that is, $\mathbf{u}_h=\bar{\mathbf{u}}_h$
and $p_h=\bar p_h$.
\end{theorem}

\begin{proof}
Let $e\in \mathcal{E}_h^0$ be an interior edge shared by $T_1$ and
$T_2$. By letting $q=0$, $\boldsymbol{\mu}=[\![\mathbf{u}_h]\!]_e$
on $e\in\partial T_1$, $\boldsymbol{\mu}=-[\![\mathbf{u}_h]\!]_e$
for $e\in
\partial T_2$, and $\boldsymbol{\mu}=0$ otherwise in (\ref{HWG2}), we have from
(\ref{bilinear_c}) that
$$
0=c(\mathbf{u}_h, \boldsymbol{\mu})=\sum_{T\in\mathcal{T}_h}\langle
\mathbf{u}_h, \boldsymbol{\mu}\rangle_{\partial T}=\int_e
[\![\mathbf{u}_h]\!]_e^2 ds,
$$
which implies $[\![\mathbf{u}_h]\!]_e=0$ for any $e\in
\mathcal{E}_h^0$.

Next by letting $\boldsymbol{\mu}=0$, we obtain from
(\ref{bilinear_c}) that
$$
b(\mathbf{u}_h, q)= 0
$$
for all $q\in W_h$.

For any $\mathbf{v}\in V_h^0$, it follows from
$[\![\mathbf{v}]\!]_e=0$ on any $e\in \mathcal{E}_h^0$ and
$\mathbf{v}=0$ on $\partial \Omega$ that
$$
c(\mathbf{v},\boldsymbol{\lambda}_h)=
\sum_{T\in\mathcal{T}_h}\langle \mathbf{v}_b,
\boldsymbol{\lambda}_h\rangle_{\partial T} =
\sum_{T\in\mathcal{T}_h}\langle [\![\mathbf{v}]\!],
\boldsymbol{\lambda}\rangle_{e} =0.
$$
Thus, we arrive at
$$
a(\mathbf{u}_h, \mathbf{v})-b(\mathbf{v}, p_h)=(\mathbf{f},
\mathbf{v}_0),
$$
which is the same as (\ref{WG1}). It implies that
$(\mathbf{u}_h;p_h)$ is a solution of the WG scheme
(\ref{WG1})-(\ref{WG2}). It follows from the uniqueness of solution
of (\ref{WG1})-(\ref{WG2}) that $\mathbf{u}_h= \bar{\mathbf{u}}_h$
and $p_h= \bar{p}_h$, which completes the proof.
\end{proof}

\section{Stability Conditions for
HWG}\label{Section:Stability_Conditions}

It is easy to see that the following defines a norm in the finite
element space $\Xi_h$
\begin{eqnarray}\label{Xi-norm}
\|\boldsymbol{\lambda}\|_{\Xi_h}=\left(\sum_{e\in\mathcal{E}_h^0}h_e
\|\boldsymbol{\lambda}\|_e^2\right)^{\frac12}.
\end{eqnarray}
As to $V_h^0$, for any $\mathbf{v}\in V_h^0$, let
\begin{eqnarray}\label{V0-norm}
\|\mathbf{v}\|_{V_h^0}=\left({|\hspace{-.02in}|\hspace{-.02in}|}
\mathbf{v}{|\hspace{-.02in}|\hspace{-.02in}|}^2+\sum_{e\in\mathcal{E}_h^0}h^{-1}_e\|[\![\mathbf{v}]\!]_e\|_e^2\right)^{\frac12}.
\end{eqnarray}
We claim that $\|\cdot\|_{V_h^0}$ defines a norm in $V_h^0$. In
fact, if $\|\mathbf{v}\|_{V_h^0}=0$, then $[\![\mathbf{v}]\!]_e=0$
on each interior edge or flat face $e\in \mathcal{E}_h^0$, and hence
$\mathbf{v}\in \mathcal{V}_h^0$. Since
${|\hspace{-.02in}|\hspace{-.02in}|}\cdot{|\hspace{-.02in}|\hspace{-.02in}|}$
defines a norm in the vector space $\mathcal{V}_h^0$, then
$\mathbf{v}=0$. This verifies the positivity property of
$\|\cdot\|_{V_h^0}$. The other properties for a norm can be checked
trivially.

\begin{remark}
Similarly, for any $\boldsymbol{\phi}=(p;\boldsymbol{\lambda})\in
M_h$, we can define
\begin{eqnarray}
\|\boldsymbol{\phi}\|_{M_h}=\|p\|+\|\boldsymbol{\lambda}\|_{\Xi_h}.
\end{eqnarray}
\end{remark}

\begin{lemma}{\rm(\cite{WY14})}\label{Trace inequality}
~\emph{\rm (}Trace Inequality{\rm)} Let $\mathcal{T}_h$ be a
partition of the domain $\Omega$ into polygons in 2D or polyhedra in
3D. Assume that the partition $\mathcal{T}_h$ satisfies the
assumptions A1, A2, and A3 as specified in \cite{WY14}. Let $p>1$ be
any real number. Then, there exists a constant $C$ such that for any
$T\in \mathcal{T}_h$ and edge/face $e\in\partial T$, we have
\begin{eqnarray}\label{Trace inequality00}\|\theta\|^p_{L^p(e)}\leq
Ch_T^{-1}(\|\theta\|^p_{L^p(T)}+h^p_T\|\nabla\theta\|^p_{L^p(T)}),
\end{eqnarray}
where $\theta\in W^{1,p}(T)$ is any function.
\end{lemma}

\begin{lemma}{\rm(\cite{WY14})}\label{Inverse Inequality}~{\rm (}Inverse Inequality{\rm )} Let
$\mathcal{T}_h$ be a partition of the domain $\Omega$ into polygons
or polyhedra. Assume that $\mathcal{T}_h$ satisfies all the
assumptions A1-A4 and $p\geq 1$ be any real number. Then, there
exists a constant $C(n)$ such that
\begin{eqnarray}\label{Inverse Inequality00}
\|\nabla\varphi\|_{T,p}\leq
C(n)h^{-1}_T\|\varphi\|_{T,p},\quad\forall T\in\mathcal{T}_h
\end{eqnarray}
for any piecewise polynomial $\varphi$ of degree $n$ on
$\mathcal{T}_h$.
\end{lemma}

\begin{lemma}\label{Boundedness}
~\emph{\rm (}Boundedness{\rm)} There exists a constant $C>0$ such
that
\begin{eqnarray}\label{Boundedness1}|a(\mathbf{w}, \mathbf{v})|&\le&
C\|\mathbf{w}\|_{V_h^0}\|\mathbf{v}\|_{V_h^0}, \quad \forall
\mathbf{w}, \mathbf{v}\in V_h^0,
\\\label{Boundedness3}
|b(\mathbf{v}, q)|&\le& C\|\mathbf{v}\|_{V_h^0}\|q\|, \quad\,\quad
\forall \mathbf{v}\in V_h^0, q\in W_h,
\\
\label{Boundedness2} |c(\mathbf{v},\boldsymbol{\lambda})|&\le& C
\|\mathbf{v}\|_{V_h^0}\|\boldsymbol{\lambda}\|_{\Xi_h},\quad \forall
\mathbf{v}\in V_h^0, \boldsymbol{\lambda}\in\Xi_h.
\end{eqnarray}
\end{lemma}

\begin{proof}
To derive (\ref{Boundedness1}), we use the Cauchy-Schwarz inequality
to obtain
\begin{equation*}
\begin{aligned}
|a(\mathbf{w},\mathbf{v})|=&\left|\sum_{T\in\mathcal{T}_h}(\nabla_w
\mathbf{w}, \nabla_w \mathbf{v})_T+h_T^{-1}\langle Q_b
\mathbf{w}_0-\mathbf{w}_b, Q_b
\mathbf{v}_0-\mathbf{v}_b\rangle_{\partial T}\right|
\\
\leq &\left(\sum_{T\in\mathcal{T}_h}\|\nabla_w
\mathbf{w}\|^2_T\right)^{\frac12}\left(\sum_{T\in\mathcal{T}_h}\|\nabla_w
\mathbf{v}\|^2_T\right)^{\frac12}
\\
&+\left(\sum_{T\in\mathcal{T}_h}h_T^{-1}\|Q_b
\mathbf{w}_0-\mathbf{w}_b\|^2_{\partial
T}\right)^{\frac12}\left(\sum_{T\in\mathcal{T}_h}h_T^{-1}\| Q_b
\mathbf{v}_0-\mathbf{v}_b\|^2_{\partial T}\right)^{\frac12}
\\
\leq &C\|\mathbf{w}\|_{V_h^0}\|\mathbf{v}\|_{V_h^0}.
\end{aligned}
\end{equation*}

 As to (\ref{Boundedness3}), we use
(\ref{wDivergence}), trace inequality (\ref{Trace inequality00}),
and inverse inequality (\ref{Inverse Inequality00}) to obtain
\begin{eqnarray*}
|b(\mathbf{v},q)|&=&\left|\sum_{T\in\mathcal{T}_h}(\nabla_w\cdot \mathbf{v}, q)_T\right|\\
&=& -\sum_{T\in\mathcal{T}_h} \left(\mathbf{v}_0,\nabla p\right)_T+\sum_{T\in\mathcal{T}_h} \langle \mathbf{v}_b, p\mathbf{n} \rangle_{\partial T}\\
&=& \sum_{T\in\mathcal{T}_h}
\left(\nabla\cdot\mathbf{v}_0,p\right)_T
 - \sum_{T\in\mathcal{T}_h} \langle \mathbf{v}_0-\mathbf{v}_b, p\mathbf{n} \rangle_{\partial T}\\
&\le& \left(\sum_{T\in\mathcal{T}_h} \|\nabla \cdot \mathbf{v}_0\|_T\right)^{\frac12}\left(\sum_{T\in\mathcal{T}_h}\|p\|^2_T\right)^{\frac12}\\
&&+ \left(\sum_{T\in\mathcal{T}_h} \|\mathbf{v}_0-\mathbf{v}_b\|_{\partial T}\right)^{\frac12}\left(\sum_{T\in\mathcal{T}_h}\|p\|^2_{\partial T}\right)^{\frac12}\\
&\le& C\left(\sum_{T\in\mathcal{T}_h} \|\nabla\mathbf{v}_0\|_T\right)^{\frac12}\|p\|\\
&&+ C h^{-\frac12}\left(\sum_{T\in\mathcal{T}_h}
\|\mathbf{v}_0-\mathbf{v}_b\|_{\partial
T}\right)^{\frac12}\left(\|p\|
+h\|\nabla p\|\right)\\
&\le& C\|\mathbf{v}\|_{V_h^0}\|p\|.
\end{eqnarray*}

As to (\ref{Boundedness2}), it follows from the Cauchy-Schwarz
inequality that
\begin{equation*}
\begin{aligned}
|c(\mathbf{v},\boldsymbol{\lambda})|=&\left|\sum_{T\in\mathcal{T}_h}\langle\mathbf{v}_b,
\boldsymbol{\lambda}\rangle_{\partial
T}\right|=\left|\sum_{e\in\mathcal{E}_h^0} \langle
[\![\mathbf{v}]\!]_e, \boldsymbol{\lambda}\rangle_e\right|
\\
\leq &\left(\sum_{e\in\mathcal{E}_h^0}h_e^{-1}\|[\![\mathbf{v}]\!]_e
\|^2_{e}\right)^{\frac12}\left(\sum_{e\in\mathcal{E}_h^0}h_e\|
\boldsymbol{\lambda} \|^2_{e}\right)^{\frac12}
\\
\leq &C\|\mathbf{v}\|_{V_h^0}\|\boldsymbol{\lambda}\|_{\Xi_h},
\end{aligned}
\end{equation*}
which completes the proof.
\end{proof}

\begin{lemma}\label{Coercivity}
~\emph{\rm (Coercivity}{\rm)} For any $\mathbf{v}\in
\mathcal{V}_h^0$, we have
\begin{eqnarray}\label{Coercivity1}|a(\mathbf{v}, \mathbf{v})|&\ge&
C\|\mathbf{v}\|^2_{V_h^0}.
\end{eqnarray}
\end{lemma}

\begin{proof}
For any $\mathbf{v}\in \mathcal{V}_h^0$, we have
$\|\mathbf{v}\|^2_{V_h^0}={|\hspace{-.02in}|\hspace{-.02in}|}
\mathbf{v}{|\hspace{-.02in}|\hspace{-.02in}|}$, which means the
estimate (\ref{Coercivity1}) holds true. This completes the proof.
\end{proof}

\begin{lemma}{\rm(\cite{WY3})}\label{REF-inf-sup}
There exists a positive constant $\beta$ independent of $h$ such
that
\begin{eqnarray}\label{REF-inf-supCond}\sup_{\mathbf{v}\in \mathcal{V}_h^0} \frac{b(\mathbf{v}, \rho)}{{|\hspace{-.02in}|\hspace{-.02in}|}\mathbf{v}{|\hspace{-.02in}|\hspace{-.02in}|}}\ge \beta
\|\rho\|,
\end{eqnarray}
for all $\rho\in W_h$.
\end{lemma}

\begin{lemma}\label{REF-inf-sup2}
For any given $\rho\in W_h$, there exist a positive constant $\beta$
independent of $h$ and a $\mathbf{v}\in V_h^0$ such that
\begin{eqnarray}\label{REF-inf-supCond2}\frac{b(\mathbf{v}, \rho)}{\|\mathbf{v}\|_{V_h^0}}\ge \beta
\|\rho\|.
\end{eqnarray}
\end{lemma}

\begin{proof}
From $\mathcal{V}_h^0\subset V_h^0$ and Lemma \ref{REF-inf-sup}, we
have, for any $\rho\in W_h$, there exists a $\mathbf{v}\in
\mathcal{V}_h^0$
\begin{eqnarray}\label{IS-inequality0}
\frac{b(\mathbf{v}, \rho)}{\|\mathbf{v}\|_{V_h^0}}=
\frac{b(\mathbf{v},
\rho)}{{|\hspace{-.02in}|\hspace{-.02in}|}\mathbf{v}{|\hspace{-.02in}|\hspace{-.02in}|}}\ge
\beta \|\rho\|,
\end{eqnarray}
which completes the proof of the lemma.
\end{proof}

\begin{lemma}\label{inf-sup}
For any given $\boldsymbol{\tau}\in \Xi_h$, there exist a
$\mathbf{v}\in V_h^0$ with $\mathbf{v}_0=\b0$ and a constant $C>0$
such that
\begin{eqnarray}\label{inf-supCond} \frac{c(\mathbf{v},
\boldsymbol{\tau})}{\|\mathbf{v}\|_{V_h^0}}\ge
C\|\boldsymbol{\tau}\|_{\Xi_h}.
\end{eqnarray}
\end{lemma}

\begin{proof}
For any $\boldsymbol{\tau}\in\Xi_h$, we have
$\langle\!\langle\boldsymbol{\tau}\rangle\!\rangle_e=0$ or
equivalently $\boldsymbol{\tau}^1+\boldsymbol{\tau}^2=0$ on each
interior edge $e\in\mathcal{E}_h^0$ and $\boldsymbol{\tau}=0$ on all
boundary edges. By letting
$\mathbf{v}=\{\b0,h_e\boldsymbol{\tau}\}\in V_h^0$ in $c(\mathbf{v},
\boldsymbol{\tau})$ and $s(\mathbf{v},\mathbf{v})$, we obtain
\begin{equation}\label{RZhang-Feb10.001}
\begin{split}
c(\mathbf{v}, \boldsymbol{\tau})= &\sum_{e\in\mathcal{E}_h^0}\langle
\mathbf{v}_b^1, \boldsymbol{\tau}^1\rangle_e+\langle \mathbf{v}_b^2,
\boldsymbol{\tau}^2\rangle_e=
2\sum_{e\in\mathcal{E}_h^0}h_e\|\boldsymbol{\tau}\|_e^2,
\end{split}
\end{equation}
and
\begin{equation}\label{RZhang-Feb10.002}
\begin{split}
s(\mathbf{v}, \mathbf{v})=
&2\sum_{e\in\mathcal{E}_h^0}h_T^{-1}h_e^2(\|\boldsymbol{\tau}^1\|_e^2+\|\boldsymbol{\tau}^2\|_e^2)
\le 2\sum_{e\in\mathcal{E}_h^0}h_e\|\boldsymbol{\tau}\|_e^2.
\end{split}
\end{equation}
It follows from (\ref{Discrete_wGradient}), Cauchy-Schwarz
inequality, the trace inequality (\ref{Trace inequality00}), and the
inverse inequality (\ref{Inverse Inequality00}) that
\begin{equation}\label{RZhang-Feb10.003}
\begin{split}
(\nabla_w\mathbf{v}, \nabla_w\mathbf{v})_T= &\sum_{e\in\partial T}\langle\mathbf{v}^*_b, \nabla_w\mathbf{v}\rangle_e\le \sum_{e\in\partial T}h_e\|\boldsymbol{\tau}^*\|_e\|\nabla_w\mathbf{v}\|_e\\
\le& C\sum_{e\in\partial
T}h^{\frac12}_e\|\boldsymbol{\tau}^*\|_e\|\nabla_w\mathbf{v}\|_T,
\end{split}
\end{equation}
where $\mathbf{v}_b^*$ is chosen to be $\mathbf{v}_b^1$ or
$\mathbf{v}_b^2$ according to the relative position of
$\mathbf{v}_b$ and $e$, which implies that
\begin{equation}\label{RZhang-Feb10.004}
\|\nabla_w\mathbf{v}\|_T\le C\sum_{e\in\partial
T}h^{\frac12}_e\|\boldsymbol{\tau}^*\|_e.
\end{equation}
By summing over all elements, we obtain
\begin{equation}\label{RZhang-Feb10.005}
(\nabla_w\mathbf{v}, \nabla_w\mathbf{v})_h\le
C\sum_{e\in\mathcal{E}_h^0 }h_e\|\boldsymbol{\tau}^*\|^2_e.
\end{equation}
It follows from (\ref{RZhang-Feb10.002}) and
(\ref{RZhang-Feb10.005}) that
\begin{equation}\label{RZhang-Feb10.006}
{|\hspace{-.02in}|\hspace{-.02in}|}
\mathbf{v}{|\hspace{-.02in}|\hspace{-.02in}|}^2\le
C\sum_{e\in\mathcal{E}_h^0
}h_e\|\boldsymbol{\tau}^*\|^2_e=C\|\boldsymbol{\tau}\|_{\Xi_h}^2.
\end{equation}
By combining (\ref{RZhang-Feb10.001}) and (\ref{RZhang-Feb10.006}),
we obtain that there exists a constant $C>0$ such that
\begin{eqnarray}\label{RZhang-Feb10.007} \frac{c(\mathbf{v},
\boldsymbol{\tau})}{\|\mathbf{v}\|_{V_h^0}}\ge
C\|\boldsymbol{\tau}\|_{\Xi_h},
\end{eqnarray}
which completes the proof.
\end{proof}

\section{Error Estimates}

The goal of this section is to derive an error equation for the HWG
finite element solution obtained from (\ref{HWG1})-(\ref{HWG2}).
This error equation is critical in convergence analysis.

In addition to the projection $Q_h=\{Q_0, Q_b\}$ defined in the
previous section, let $\mathbb{Q}_h$ and $\mathbf{Q}_h$ be two local
$L^2$ projections onto $P_{k-1}(T)$ and $[P_{k-1}(T)]^{d\times d}$,
respectively.

\begin{lemma}{\rm(\cite{WY3})}\label{commutative_properties}
The projection operators $Q_h$, $\mathbf{Q}_h$, and $\mathbb{Q}_h$
satisfy the following commutative properties
\begin{eqnarray}\label{CP1}\nabla_w(Q_h\mathbf{v})&=&\mathbf{Q}_h(\nabla \mathbf{v}),
\quad \forall \mathbf{v}\in [H^1(\Omega)]^d,
\\
\label{CP2}\nabla_w\cdot (Q_h\mathbf{v})&=&\mathbb{Q}_h(\nabla
\cdot\mathbf{v}), \quad \forall \mathbf{v}\in H({\rm div},\Omega).
\end{eqnarray}
\end{lemma}

Denote by $(\mathbf{u}; p)$ the exact solution of
(\ref{Stokes})-(\ref{boundarycond}). Let $(\mathbf{u}_h; p_h;
\boldsymbol{\lambda}_h)\in V_h\times W_h\times \Xi_h$ be the
solutions of (\ref{HWG1})-(\ref{HWG2}). Let
$\boldsymbol{\lambda}=\nabla \mathbf{u}\cdot
\mathbf{n}-p\mathbf{n}$. Define error functions as follows
\begin{eqnarray}\label{errors-01}
\mathbf{e}_h=\{Q_0\mathbf{u}-\mathbf{u}_0, Q_b
\mathbf{u}-\mathbf{u}_b\}, \quad \varepsilon_h=\mathbb{Q}_h p-p_h,
\quad \boldsymbol{\delta}_h=Q_b
\boldsymbol{\lambda}-\boldsymbol{\lambda}_h.
\end{eqnarray}

\begin{lemma}\label{error-equation}
Let $(\mathbf{u}; p)$ be the exact solution of
(\ref{Stokes})-(\ref{boundarycond}), and $(\mathbf{u}_h; p_h;
\boldsymbol{\lambda}_h)\in V_h\times W_h\times \Xi_h$ be the
solutions of (\ref{HWG1})-(\ref{HWG2}). Then, the error functions
$\mathbf{e}_h$, $\varepsilon_h$, and $\delta_h$ satisfy the
following equations
\begin{eqnarray}\label{ee1}a(\mathbf{e}_h, \mathbf{v})-b(\mathbf{v}, \varepsilon_h)-c(\mathbf{v},
\boldsymbol{\delta}_h)&=&\ell_{\mathbf{u},p}(\mathbf{v}), \quad
\forall \mathbf{v}\in V_h^0,
\\
\label{ee2}b(\mathbf{e}_h, q)+c(\mathbf{e}_h,\boldsymbol{\mu})&=&0,
\qquad\ \ \forall q\in W_h, \boldsymbol{\mu}\in\Xi_h,
\end{eqnarray}
\end{lemma}
where
\begin{equation}\label{error term_ell}
\begin{split}
\ell_{\mathbf{u},p}(\mathbf{v}) = &\sum_{T\in\mathcal{T}_h}\langle
\mathbf{v}_0 - \mathbf{v}_b, (\nabla\mathbf{u} - \mathbf{Q}_h(\nabla
\mathbf{u}))\cdot \mathbf{n} \rangle_{\partial T}
\\ &-\sum_{T\in\mathcal{T}_h} \langle \mathbf{v}_0-\mathbf{v}_b , (p-\mathbb{Q}_h p) \mathbf{n} \rangle_{\partial T}
+s(Q_h\mathbf{u},\mathbf{v}).
\end{split}
\end{equation}

\begin{proof}
First, applying (\ref{wGradient}), Lemma
\ref{commutative_properties}, and the integration by parts, we have
\begin{eqnarray}
\nonumber(\nabla_w(Q_h \mathbf{u}),\nabla_w \mathbf{v})_T &=& (\mathbf{Q}_h(\nabla \mathbf{u}),\nabla_w \mathbf{v})_T\\
\nonumber&=& -(\mathbf{v}_0,\nabla\cdot\mathbf{Q}_h(\nabla
\mathbf{u}))_T+\langle \mathbf{v}_b, \mathbf{Q}_h(\nabla \mathbf{u})
\cdot \mathbf{n}\rangle_{\partial T}\\
\nonumber&=& (\nabla \mathbf{v}_0,\mathbf{Q}_h(\nabla
\mathbf{u}))_T-\langle \mathbf{v}_0-\mathbf{v}_b,
\mathbf{Q}_h(\nabla \mathbf{u}) \cdot \mathbf{n} \rangle_{\partial T}\\
\nonumber&=& (\nabla \mathbf{v}_0,\nabla \mathbf{u})_T-\langle
\mathbf{v}_0-\mathbf{v}_b,
\mathbf{Q}_h(\nabla \mathbf{u}) \cdot \mathbf{n} \rangle_{\partial T}\\
\nonumber&=& -(\Delta \mathbf{u}, \mathbf{v}_0)_T+\langle
\mathbf{v}_0 -
\mathbf{v}_b, (\nabla\mathbf{u} - \mathbf{Q}_h(\nabla \mathbf{u}))\cdot \mathbf{n} \rangle_{\partial T}\\
\nonumber&&+\langle \mathbf{v}_b , \nabla \mathbf{u} \cdot
\mathbf{n} \rangle_{\partial T}.
\end{eqnarray}
Summing over all $T\in \mathcal{T}_h$ reaches
\begin{eqnarray}\label{error-eq1}
\nonumber-(\Delta \mathbf{u}, \mathbf{v}_0)&=&(\nabla_w(Q_h
\mathbf{u}),\nabla_w \mathbf{v})-\sum_{T\in\mathcal{T}_h}\langle
\mathbf{v}_0 -
\mathbf{v}_b, (\nabla\mathbf{u} - \mathbf{Q}_h(\nabla \mathbf{u}))\cdot \mathbf{n} \rangle_{\partial T}\\
&&-\sum_{T\in\mathcal{T}_h}\langle \mathbf{v}_b , \nabla \mathbf{u}
\cdot \mathbf{n} \rangle_{\partial T}.
\end{eqnarray}

Similarly, by using (\ref{wDivergence}) and the integration by
parts, we have
\begin{eqnarray}
\nonumber (\nabla_w \cdot \mathbf{v}, \mathbb{Q}_h p)_T&=&
-(\mathbf{v}_0, \nabla(\mathbb{Q}_h p))_T+\langle \mathbf{v}_b, (\mathbb{Q}_h p)\mathbf{n}\rangle_{\partial T}\\
\nonumber&=& (\nabla \cdot \mathbf{v}_0, \mathbb{Q}_h p)_T-\langle
\mathbf{v}_0-\mathbf{v}_b,
 (\mathbb{Q}_h p)\mathbf{n}\rangle_{\partial T}\\
\nonumber&=& (\nabla \cdot \mathbf{v}_0,  p)_T-\langle
\mathbf{v}_0-\mathbf{v}_b,
 (\mathbb{Q}_h p)\mathbf{n}\rangle_{\partial T}\\
\nonumber&=& -(\mathbf{v}_0, \nabla p)_T + \langle \mathbf{v}_0,
p\mathbf{n} \rangle_{\partial T}
-\langle \mathbf{v}_0-\mathbf{v}_b,(\mathbb{Q}_h p)\mathbf{n}\rangle_{\partial T}\\
\nonumber&=& -(\mathbf{v}_0, \nabla p)_T + \langle
\mathbf{v}_0-\mathbf{v}_b,(p-\mathbb{Q}_h
p)\mathbf{n}\rangle_{\partial T} + \langle \mathbf{v}_b, \nabla
\mathbf{u}\cdot \mathbf{n} \rangle_{\partial T},
\end{eqnarray}
Summing over all $T$ leads to
\begin{eqnarray}\label{error-eq2}
\nonumber( \nabla p, \mathbf{v}_0)&=&(\nabla_w \cdot \mathbf{v},
\mathbb{Q}_h p) +
\sum_{T\in\mathcal{T}_h}\langle \mathbf{v}_0-\mathbf{v}_b,(p-\mathbb{Q}_h p)\mathbf{n}\rangle_{\partial T}\\
&&+ \sum_{T\in\mathcal{T}_h}\langle \mathbf{v}_b, p \mathbf{n}
\rangle_{\partial T}.
\end{eqnarray}

By using the identity $-(\Delta \mathbf{u}, \mathbf{v}_0)+(\nabla p,
\mathbf{v}_0)=(\mathbf{f},\mathbf{v}_0)$ and noticing that
\begin{eqnarray*}
\sum_{T\in\mathcal{T}_h} \langle \mathbf{v}_b,
\nabla\mathbf{u}\cdot\mathbf{n}-p\mathbf{n}\rangle_{\partial
T}=c(\mathbf{v},\boldsymbol{\lambda}),
\end{eqnarray*}
we obtain
\begin{eqnarray}\label{error-eq4}
a(\mathbf{v}, \mathbf{Q}_h \mathbf{u}) - b(\mathbf{v}, \mathbb{Q}_h
p) -c(\mathbf{v}, \boldsymbol{\lambda}) = (\mathbf{f},
\mathbf{v}_0)+ \ell_{\mathbf{u},p}(\mathbf{v}).
\end{eqnarray}

Combining with the scheme (\ref{HWG1}) as follows
\begin{eqnarray*}
a(\mathbf{v}, \mathbf{u}_h) - b(\mathbf{v}, p_h) -c(\mathbf{v},
\boldsymbol{\lambda}_h) = (\mathbf{f}, \mathbf{v}_0),
\end{eqnarray*}
we obtain
\begin{eqnarray*}
a(\mathbf{e}_h, \mathbf{v})-b(\mathbf{v},
\varepsilon_h)-c(\mathbf{v},
\boldsymbol{\delta}_h)&=&\ell_{\mathbf{u},p}(\mathbf{v}).
\end{eqnarray*}

As to (\ref{ee2}), from Theorem \ref{theorem:wg-hwg} we know that
$[\![\mathbf{e}_h]\!]=0$, which leads to
\begin{eqnarray*}
c(\mathbf{e}_h, \boldsymbol{\mu})=0, \quad \forall
\boldsymbol{\mu}\in\Xi_h.
\end{eqnarray*}
Moreover, for any $q\in W_h$, we have
\begin{eqnarray*}
b(\mathbf{e}_h, q)&=& b(\mathbf{Q}_h \mathbf{u},q)\\
&=&\sum_{T\in\mathcal{T}_h} (\nabla_w \cdot (\mathbf{Q}_h \mathbf{u}),q)_T\\
&=&\sum_{T\in\mathcal{T}_h} (\mathbb{Q}_h (\nabla \cdot \mathbf{u}), q)_T\\
&=&(\nabla \cdot \mathbf{u}, q)=0.
\end{eqnarray*}
This completes the proof.
\end{proof}

Next we shall establish some error estimates for the hybridized WG
finite element solution $(\mathbf{u}_h; p_h; \boldsymbol{\lambda}_h
)$ arising from (\ref{HWG1})-(\ref{HWG2}). The error equations
(\ref{ee1})-(\ref{ee2}) imply
\begin{eqnarray*}a(Q_h \mathbf{u}-\mathbf{u}_h, \mathbf{v})-b(\mathbf{v}, \mathbb{Q}_h p-p_h)-c(\mathbf{v},
Q_b
\boldsymbol{\lambda}-\boldsymbol{\lambda}_h)&=&\ell_{\mathbf{u},p}(\mathbf{v}),
\quad \forall \mathbf{v}\in V_h^0,
\\
b(Q_h \mathbf{u}-\mathbf{u}_h, q)+c(Q_h
\mathbf{u}-\mathbf{u}_h,\boldsymbol{\mu})&=&0, \quad \quad \ \ \
\forall q\in W_h, \boldsymbol{\mu}\in\Xi_h,
\end{eqnarray*}
where $\ell_{\mathbf{u},p}(\mathbf{v})$ is given by (\ref{error
term_ell}). The above is a saddle point problem for which the
Brezzi's theorem \cite{B74} can be applied for an analysis on its
stability and solvability. Note that all the conditions of Brezzi¡¯s
theorem have been verified in Section
\ref{Section:Stability_Conditions} (see Lemmas \ref{Boundedness},
\ref{Coercivity}, and \ref{inf-sup}). The following error estimate
can be proved similarly with Theorem 7.1 of \cite{WY3}.

\begin{theorem}\label{e-estimate1}
Let $(\mathbf{u}; p)$ be the exact solution of
(\ref{Stokes})-(\ref{boundarycond}) and $(\mathbf{u}_h; p_h;
\boldsymbol{\lambda}_h)\in V_h\times W_h\times \Xi_h$ be the
solutions of (\ref{HWG1})-(\ref{HWG2}).  Then, there exists a
constant $C$ such that
\begin{equation}\label{eer1}
\|Q_h\mathbf{u}-\mathbf{u}_h\|_{V_h^0}+\|\mathbb{Q}_h p-p_h\|+\|Q_b
\boldsymbol{\lambda}-\boldsymbol{\lambda}_h\|_{\Xi_h}\le C h^k
(\|\mathbf{u}\|_{k+1}+\|p\|_k).
\end{equation}
\end{theorem}

\begin{theorem}\label{e-estimate2}
Let $(\mathbf{u}; p)$ be the exact solution of
(\ref{Stokes})-(\ref{boundarycond}) and $\boldsymbol{\lambda}_h\in
\Xi_h$ be the last component of the solution of
(\ref{HWG1})-(\ref{HWG2}). On the set of interior edges
$\mathcal{E}_h^0$, let $\boldsymbol{\lambda}=\nabla \mathbf{u}\cdot
\mathbf{n} - p\mathbf{n}$.
 Then, there
exists a constant $C$ such that
\begin{equation}\label{eer2}
\|\boldsymbol{\lambda}-\boldsymbol{\lambda}_h\|_{\Xi_h}\le C h^k
(\|\mathbf{u}\|_{k+1}+\|p\|_k).
\end{equation}
\end{theorem}

\begin{proof}
From the triangle inequality and Theorem \ref{e-estimate1}, we have
\begin{eqnarray*}
\|\boldsymbol{\lambda}-\boldsymbol{\lambda}_h\|_{\Xi_h}&\le&
\|\boldsymbol{\lambda}-Q_b \boldsymbol{\lambda}\|_{\Xi_h}
+\|Q_b \boldsymbol{\lambda} - \boldsymbol{\lambda}_h\|_{\Xi_h},\\
\|Q_b \boldsymbol{\lambda} - \boldsymbol{\lambda}_h\|_{\Xi_h}&\le&
Ch^k(\|u\|_{k+1}+\|p\|_k).
\end{eqnarray*}
Thus we just need to concentrate on $\|\boldsymbol{\lambda}-Q_b
\boldsymbol{\lambda}\|_{\Xi_h}$.

Applying Definition \ref{Xi-norm}, trace inequality,
 and the property of $L^2$ projection,
yields
\begin{eqnarray*}
\|\boldsymbol{\lambda}-Q_b \boldsymbol{\lambda}\|_{\Xi_h}^2&=&\sum_{e\in\mathcal{E}_h^0} h_e\|\boldsymbol{\lambda}-Q_b \boldsymbol{\lambda}\|_e^2\\
&\le& \sum_{e\in\mathcal{E}_h^0} h\|\nabla \mathbf{u}-\mathbf{Q}_h \nabla \mathbf{u}\|^2_e+\sum_{e\in\mathcal{E}_h^0} h\|p-\mathbb{Q}_h p\|^2_e\\
&\le& C\sum_{T\in\mathcal{T}_h} (\|\nabla \mathbf{u}-\mathbf{Q}_h \nabla \mathbf{u}\|^2_T+ h^2\|\nabla \mathbf{u}-\mathbf{Q}_h \nabla \mathbf{u}\|^2_{1,T})\\
&& + C\sum_{T\in\mathcal{T}_h} (\|p-\mathbb{Q}_h p\|^2_T+h^2 \|p-\mathbb{Q}_h p\|^2_{1,T})\\
&\le& Ch^{2k}(\|u\|_{k+1}+\|p\|_k)^2,
\end{eqnarray*}
which completes the proof.
\end{proof}

The following $L^2$-error estimate for $Q_0\mathbf{u}-\mathbf{u}_0$
follows from Theorem \ref{theorem:wg-hwg} and Theorem 7.2 of
\cite{WY3}.

\begin{theorem}{\rm(\cite{WY3})}\label{e-estimate3}
Let $(\mathbf{u}; p)$ with $k\ge 1$ and $(\mathbf{u}_h; p_h;
\boldsymbol{\lambda}_h)\in V_h\times W_h\times \Xi_h$ be the exact
solution of (\ref{Stokes})-(\ref{boundarycond}) and  be the
solutions of (\ref{HWG1})-(\ref{HWG2}), respectively.  Then, the
following optimal order error estimate holds true
\begin{equation}\label{eer4}
\|Q_0\mathbf{u}-\mathbf{u}_0\| \le C h^{k+1}
(\|\mathbf{u}\|_{k+1}+\|p\|_k).
\end{equation}
\end{theorem}

\section{Efficient Implementation via Variable
Reduction}\label{Variable_Recuction}

The degrees of freedom in the WG algorithm (\ref{WG1})-(\ref{WG2})
can be divided into two classes: (1) the interior variables
representing $\mathbf{u}_0$, and (2) the interface variables for
$\mathbf{u}_b$. For the HWG algorithm (\ref{HWG1})-(\ref{HWG2}),
more unknowns must be added to the picture from the Lagrange
multiplier $\boldsymbol{\lambda}_h$. Thus, the size of the discrete
system arising from either (\ref{WG1})-(\ref{WG2}) or
(\ref{HWG1})-(\ref{HWG2}) is enormously large.

The goal of this section is to present a Schur complement
formulation for the WG algorithm (\ref{WG1})-(\ref{WG2}) based on
the hybridized formulation (\ref{HWG1})-(\ref{HWG2}). The method
shall eliminate all the interior unknowns associated with
$\mathbf{u}_0$ and the interface unknow $\boldsymbol{\lambda}_h$,
and produce a much reduced system of linear equations involving only
the unknowns representing the interface variables $\mathbf{u}_b$.


\subsection{Theory of variable reduction}\label{SubVariable_Recuction}

Denote by $B_h$ the interface finite element space defined as the
restriction of $\mathcal{V}_h$ on the set of edges $\mathcal{E}_h$;
i.e.,
$$
B_h=\{\mathbf{v}=\{\boldsymbol{\mu}; p\}: \boldsymbol{\mu}\in
[P_{k-1}(e)]^d, p|_e\in P_{k-1}(e), e\in \mathcal{E}_h\}.
$$
$B_h$ is a Hilbert space with the following inner product
$$
\langle \mathbf{w}_b, \mathbf{q}_b\rangle_{\mathcal{E}_h}=\sum_{e\in
\mathcal{E}_h}\langle \mathbf{w}_b, \mathbf{q}_b\rangle_e, \quad
\forall\mathbf{w}_b, \mathbf{q}_b\in B_h.
$$
Let $B_h^0$ be the subspace of $B_h$ consisting of functions with
vanishing boundary value on $\partial \Omega$. It is not hard to see
that the interface finite element space $B_h$ is isomorphic to the
space of Lagrange multiplier $\Xi_h$. The Schur complement through
an elimination of the Lagrange multiplier $\boldsymbol{\lambda}_h$
and the interior unknown $\mathbf{u}_0$ can be implemented through a
map, denoted by $S_\mathbf{f}$.

We define the map $S_\mathbf{f}: B_h\rightarrow B_h^0$ as follows:
for a fixed $p_h$ and any given function $\mathbf{w}_b\in B_h$, the
image $S_{\mathbf{f}}(\mathbf{w}_b; p_h)$ can be obtained by
\begin{enumerate}
\item[\textbf{Step 1.}] On each element $T\in \mathcal{T}_h$, solve for
$\mathbf{w}_0$ in term of $\mathbf{w}_b$ and $p_h$ from the
following local problem:
\begin{equation}\label{W0}
a_T(\mathbf{w}_h,\mathbf{v})-b_T(\mathbf{v},
p_h)=(\mathbf{f},\mathbf{v}_0)_T, \quad \forall
\mathbf{v}=\{\mathbf{v}_0, \b0\}\in V_k(T),
\end{equation}
where $\mathbf{w}_h=\{\mathbf{w}_0, \mathbf{w}_b\}\in  V_k(T),
p_h\in W_k(T)$. Denote $\mathbf{w}_0=D_{\mathbf{f}}(\mathbf{w}_b;
p_h)$.
\item[\textbf{Step 2.}]  On each element $T\in \mathcal{T}_h$, solve for
$\boldsymbol{\zeta}_{h,T}\in \Lambda_{k}(\partial T)$ in term of
$\mathbf{w}_h=\{\mathbf{w}_0, \mathbf{w}_b\}$ and $p_h$ from the
following local problem:
\begin{equation}\label{zetah}
c_T(\mathbf{v},
\boldsymbol{\zeta}_{h,T})=a_T(\mathbf{w}_h,\mathbf{v})-b_T(\mathbf{v},
p_h), \quad \forall \mathbf{v}=\{\b0, \mathbf{v}_b\}\in V_k(T),
\end{equation}
Thus we obtain a function $\boldsymbol{\zeta}_{h,T}\in\Lambda_h$.
Denote $\boldsymbol{\zeta}_{h,T}=L_{\mathbf{f}}(\mathbf{w}_b; p_h)$.
\item[\textbf{Step 3.}] Define $S_{\mathbf{f}}(\mathbf{w}_b; p_h)$ by the
similarity of $\boldsymbol{\zeta}_{h}$ on interior edges and zero on
boundary edges; i.e.,
\begin{equation}\label{VRSimilarity}
S_{\mathbf{f}}(\mathbf{w}_b; p_h)=\langle\!\langle
\boldsymbol{\zeta}_h\rangle\!\rangle.
\end{equation}
\end{enumerate}
It follows from (\ref{VRSimilarity}) that
$S_{\mathbf{f}}(\mathbf{w}_b; p_h)\in B_h^0$. The following are two
properties regarding the operator $S_{\mathbf{f}}$ and the related
terms:
\begin{enumerate}
\item[\textbf{(1)\ }] The sum of (\ref{W0}) and (\ref{zetah}) yields
\begin{equation}\label{W0+zetah}
c_T(\mathbf{v},
\boldsymbol{\zeta}_{h,T})=a_T(\mathbf{w}_h,\mathbf{v})-b_T(\mathbf{v},
p_h)-(\mathbf{f},\mathbf{v}_0)_T, \quad \forall
\mathbf{v}=\{\mathbf{v}_0, \mathbf{v}_b\}\in V_k(T).
\end{equation}

\item[\textbf{(2)\ }] It follows from the superposition principle
we have the following identify
\begin{equation}\label{Superposition}
S_{\mathbf{f}}(\mathbf{w}_b; p_h)=S_{\b0}(\mathbf{w}_b;
p_h)+S_{\mathbf{f}}(\b0; 0), \quad \forall \mathbf{w}_b\in B_h,
p_h\in W_h,
\end{equation}
where $S_{\b0}$ is the operator with respect to $\mathbf{f}=\b0$.
\end{enumerate}


\begin{lemma}\label{Lemma:Lemma8.0}
The following identity holds true for the operator $S_0$:
\begin{equation}\label{LinearO}
\langle S_{\b0}(\mathbf{w}_b; p_h),
\mathbf{q}_b\rangle_{\mathcal{E}_h}=a(\mathbf{w}_h,
\mathbf{q}_h)-b(\mathbf{q}_h, p_h), \quad \forall \mathbf{w}_b,
\mathbf{q}_b\in B_h^0,
\end{equation}
where $\mathbf{w}_h=\{D_{\b0}(\mathbf{w}_b; p_h), \mathbf{w}_b\}$
and $\mathbf{q}_h=\{D_{\b0}(\mathbf{q}_b; p_h), \mathbf{q}_b\}$.
\end{lemma}

\begin{proof}For any $\mathbf{w}_b, \mathbf{q}_b\in B_h^0$, from the definition of
the operator $S_{\mathbf{f}}$ we obtain
\begin{eqnarray*}
\mathbf{w}_h&=\{D_0(\mathbf{w}_b; p_h), \mathbf{w}_b\},\quad
\boldsymbol{\zeta}_h&=L_{\b0}(\mathbf{w}_b;
p_h),\\
\mathbf{q}_h&=\{D_0(\mathbf{q}_b; p_h), \mathbf{q}_b\}.\quad
\end{eqnarray*}
Letting $\mathbf{f}=\b0$ in (\ref{W0+zetah}) yields
\begin{eqnarray*}
\langle S_{\b0}(\mathbf{w}_b; p_h),
\mathbf{q}_b\rangle_{\mathcal{E}_h}&=& \sum_{e\in\mathcal{E}_h^0}
\langle \langle\!\langle \boldsymbol{\zeta}_h\rangle\!\rangle_e,
\mathbf{q}_b\rangle_{e}
\\
&=& \sum_{T\in\mathcal{T}_h} \langle  \boldsymbol{\zeta}_{h,T},
\mathbf{q}_b\rangle_{\partial T}
\\
&=& \sum_{T\in\mathcal{T}_h} c_T(\mathbf{q}_h,
\boldsymbol{\zeta}_{h,T})
\\
&=& \sum_{T\in\mathcal{T}_h} a_T(\mathbf{w}_h,
\mathbf{q}_h)-b_T(\mathbf{q}_h, p_h).
\end{eqnarray*}
This completes the identity (\ref{LinearO}).
\end{proof}

\begin{lemma}\label{Lemma:Lemma8.1}
Let $\mathbf{u}_h=\{\mathbf{u}_0,\mathbf{u}_b\}\in V_h$, $p_h\in
W_h$, and $\boldsymbol{\lambda}_h\in\Xi_h$ be the unique solution of
the hybridized WG algorithm (\ref{HWG1})-(\ref{HWG2}). Then, we have
$\mathbf{u}_h\in \mathcal{V}_h$ and $\mathbf{u}_b\in B_h$ is a well
defined function. Furthermore, it satisfies the following equation
\begin{equation}\label{VRSimilarity0}
S_{\mathbf{f}}(\mathbf{u}_b; p_h)=\langle\!\langle
\boldsymbol{\zeta}_h\rangle\!\rangle=\b0.
\end{equation}
\end{lemma}

\begin{proof}
Since $(\mathbf{u}_h; p_h;\boldsymbol{\lambda}_h)$ is the unique
solution of the HWG scheme (\ref{HWG1})-(\ref{HWG2}), then from
Theorem \ref{theorem:wg-hwg} we have $[\![\mathbf{u}_h]\!]_e=0$ on
each interior edge or flat face $e\in\mathcal{E}_h^0$ and
$\mathbf{u}_b=Q_b\textbf{g}$ on $\partial \Omega$. Thus,
$\mathbf{u}_h\in \mathcal{V}_h$ and its restriction on
$\mathcal{E}_h$ is a well-defined function in $B_h$.

In order to verify (\ref{VRSimilarity0}), we take
$\mathbf{v}=\{\mathbf{v}_0,\b0\}\in V_k(T)$ on $T$ and zero
elsewhere in (\ref{HWG1}), it follows that $\mathbf{u}_h$ satisfies
the local equation
\begin{equation*}
a_T(\mathbf{u}_h,\mathbf{v})-b_T(\mathbf{v},
p_h)=(\mathbf{f},\mathbf{v}_0)_T, \quad \forall
\mathbf{v}=\{\mathbf{v}_0, \b0\}\in V_k(T).
\end{equation*}

Next, taking $\mathbf{v}=\{\b0,\mathbf{v}_b\}\in V_k(T)$ on $T$ and
zero elsewhere in (\ref{HWG1}), yields that $\boldsymbol{\lambda}_h$
satisfies the local equation
\begin{equation*}
c_T(\mathbf{v},
\boldsymbol{\lambda}_{h,T})=a_T(\mathbf{w}_h,\mathbf{v})-b_T(\mathbf{v},
p_h), \quad \forall \mathbf{v}=\{\b0, \mathbf{v}_b\}\in V_k(T),
\end{equation*}
where $\boldsymbol{\lambda}_{h,T}$ is the restriction of
$\boldsymbol{\lambda}_h$ on $\partial T$. Thus, from the definition
of the operator $S_{\mathbf{f}}$, we obtain
\begin{equation*}
S_{\mathbf{f}}(\mathbf{u}_b; p_h)=\langle\!\langle
\boldsymbol{\lambda}_h\rangle\!\rangle.
\end{equation*}
Combining with the fact that $\boldsymbol{\lambda}_h\in\Xi_h$, we
have $\langle\!\langle \boldsymbol{\zeta}_h\rangle\!\rangle=\b0$,
which completes the proof of the lemma.
\end{proof}

\begin{lemma}\label{Lemma:Lemma8.2}
Let $\bar{\mathbf{u}}_b\in B_h$ be a function satisfying
$\bar{\mathbf{u}}_b=Q_b\textbf{g}$ on $\partial \Omega$,
$\bar{\mathbf{u}}_b$ and $p_h$ satisfy the following operator
equation
\begin{equation}\label{VRSimilarity02}
S_{\mathbf{f}}(\bar{\mathbf{u}}_b; p_h)=\b0.
\end{equation}
Then, $(\bar{\mathbf{u}}_h; p_h)\in V_h\times W_h$ is the solution
of the WG finite element solution arising from
(\ref{WG1})-(\ref{WG2}). Here $\bar{\mathbf{u}}_0$ is the solution
of the following local problems on each element $T\in\mathcal{T}_h$,
\begin{equation}\label{W0-new}
a_T(\bar{\mathbf{u}}_h,\mathbf{v})-b_T(\mathbf{v},
p_h)=(\mathbf{f},\mathbf{v}_0)_T, \quad \forall
\mathbf{v}=\{\mathbf{v}_0, \b0\}\in V_k(T),
\end{equation}
with $\bar{\mathbf{u}}_h=\{\bar{\mathbf{u}}_0,\bar{\mathbf{u}}_b\}$.
\end{lemma}

\begin{proof} For each $T\in \mathcal{T}_h$, we solve for $\bar{\boldsymbol{\lambda}}_{h,T}\in
\Lambda_k(\partial T)$ from the local problem
\begin{equation}\label{zetah-new}
c_T(\mathbf{v},
\bar{\boldsymbol{\lambda}}_{h,T})=a_T(\bar{\mathbf{u}}_h,\mathbf{v})-b_T(\mathbf{v},
p_h), \quad \forall \mathbf{v}=\{\b0, \mathbf{v}_b\}\in V_k(T).
\end{equation}
Define a function $\bar{\boldsymbol{\lambda}}_h\in \Lambda_h$ by
$\bar{\boldsymbol{\lambda}}_h|_{\partial
T}=\bar{\boldsymbol{\lambda}}_{h,T}$. Since $(\bar{\mathbf{u}}_b;
p_h)\in B_h\times W_h$ satisfies the operator equation
(\ref{VRSimilarity02}), $\bar{\mathbf{u}}_b$ satisfies the boundary
condition, and $\bar{\mathbf{u}}_0$ is given by (\ref{W0-new}), it
follows from the definition of the operator $S_{\mathbf{f}}$ that
\begin{equation}\label{Equation8-1}
\langle\!\langle\bar{\boldsymbol{\lambda}}_h\rangle\!\rangle=S_{\mathbf{f}}(\bar{\mathbf{u}}_b;
p_h)=0,
\end{equation}
which implies $\bar{\boldsymbol{\lambda}}_h\in \Xi_h$.

By subtracting (\ref{zetah-new}) from (\ref{W0-new}), we have
\begin{equation}\label{W0+zetah-new}
a_T(\bar{\mathbf{u}}_h,\mathbf{v})-b_T(\mathbf{v},
p_h)-c_T(\mathbf{v},
\bar{\boldsymbol{\lambda}}_{h,T})=(\mathbf{f},\mathbf{v}_0)_T, \quad
\forall \mathbf{v}=\{\mathbf{v}_0, \mathbf{v}_b\}\in V_k(T).
\end{equation}
By summing up the above equations over all $T\in \mathcal{T}_h$, we
obtain
\begin{equation}\label{W0+zetah-new1}
a(\bar{\mathbf{u}}_h,\mathbf{v})-b(\mathbf{v}, p_h)-c(\mathbf{v},
\bar{\boldsymbol{\lambda}}_{h})=(\mathbf{f},\mathbf{v}_0), \quad
\forall \mathbf{v}=\{\mathbf{v}_0, \mathbf{v}_b\}\in V_h.
\end{equation}
By restricting $\mathbf{v}$ to the weak finite element space $V_h^0$
and using (\ref{Equation8-1}) we arrive at
$$
c(\mathbf{v}, \bar{\boldsymbol{\lambda}}_{h})=\sum_{e\in
\mathcal{E}_h^0}\langle \langle\!\langle
\bar{\boldsymbol{\lambda}}_h\rangle\!\rangle_e,
\mathbf{v}_b\rangle_e =0.
$$
Thus we obtain
$$
a(\bar{\mathbf{u}}_h,\mathbf{v})-b(\mathbf{v},
p_h)=(\mathbf{f},\mathbf{v}_0), \quad \forall
\mathbf{v}=\{\mathbf{v}_0, \mathbf{v}_b\}\in V_h^0.
$$
Recall that the assumption $\bar{\mathbf{u}}_b=Q_b\textbf{g}$ on
$\partial \Omega$ and Theorem \ref{theorem:wg-hwg}, we have
$\bar{\mathbf{u}}_h$ is the WG finite element solution of
(\ref{WG1})-(\ref{WG2}). This completes the proof of the lemma.
\end{proof}

Combining the above two lemmas yields the following result.
\begin{theorem}\label{Thm:Thm8.3}
Let $\bar{\mathbf{u}}_b\in B_h$ be any function such that
$\bar{\mathbf{u}}_b=Q_b\textbf{g}$ on $\partial \Omega$. Let
$\bar{\mathbf{u}}_0$ be the solution of (\ref{W0-new}). Then
$(\bar{\mathbf{u}}_h; p_h)$ is the solution of
(\ref{WG1})-(\ref{WG2}) if and only if $\bar{\mathbf{u}}_b$
satisfies the following operator equation
\begin{equation}\label{Eqn8.4}
S_{\mathbf{f}}(\bar{\mathbf{u}}_b; p_h)=\b0.
\end{equation}
\end{theorem}

\subsection{Computational algorithm with reduced variables}

Together with the equation (\ref{Superposition}), (\ref{Eqn8.4})
gives rise to
\begin{equation}\label{Eqn8.5}
S_{\b0}(\bar{\mathbf{u}}_b; p_h)=-S_{\mathbf{f}}(\b0; 0).
\end{equation}
Let $\textbf{G}_b\in B_h$ be a finite element function such that
$\textbf{G}_b=Q_b\textbf{g}$ on the boundary of $\Omega$ and zero
elsewhere. From the linearity of operator $S_{\b0}$, we have
$$
S_{\b0}(\bar{\mathbf{u}}_b;
p_h)=S_{\b0}(\bar{\mathbf{u}}_b-\textbf{G}_b;
p_h)+S_{\b0}(\textbf{G}_b; p_h).
$$
Substituting this equation into (\ref{Eqn8.5}), one obtains
$$
S_{\b0}(\bar{\mathbf{u}}_b-\textbf{G}_b; p_h)=-S_{\mathbf{f}}(\b0;
0)-S_{\b0}(\textbf{G}_b; p_h).
$$
Note that the function
$\textbf{H}_b=\bar{\mathbf{u}}_b-\textbf{G}_b$ has vanishing
boundary value. Letting $\textbf{r}_b=-S_{\mathbf{f}}(\b0;
0)-S_{\b0}(\textbf{G}_b; p_h)$, we have
\begin{equation}\label{Eqn8.6}
S_{\b0}(\textbf{H}_b; p_h)=\textbf{r}_b.
\end{equation}
The reduced system of linear equations (\ref{Eqn8.6}) is actually a
Schur complement formulation for the WG finite element scheme
(\ref{WG1})-(\ref{WG2}). Note that (\ref{Eqn8.6}) involves only the
variables representing the value of the function on
$\mathcal{E}^0_h$. This is clearly a significant reduction on the
size of the linear system that has to be solved in the WG finite
element method.

\begin{algorithm3}
The solution $(\mathbf{u}_h; p_h)$ to the WG algorithm
(\ref{WG1})-(\ref{WG2}) can be obtained in the following steps:
\begin{enumerate}
\item[\textbf{Step 1.}] On each element $T\in \mathcal{T}_h$, solve for
$\textbf{r}_b$ from the following equation:
\begin{equation*}
\textbf{r}_b=-S_{\mathbf{f}}(\b0; 0)-S_{\b0}(\textbf{G}_b; p_h)
\end{equation*}
This step requires the inversion of local stiffness matrices and can
be accomplished in parallel. The computational complexity is linear
with respect to the number of unknowns.
\item[\textbf{Step 2.}]  Solve for
$\{\textbf{H}_b, p_h\}$ by means of the operator equation
(\ref{Eqn8.6}).
\item[\textbf{Step 3.}] Compute $\mathbf{u}_b=\textbf{G}_b+\textbf{H}_b$ to
get the solution on element boundaries. Then on each element $T$,
compute  $\mathbf{u}_0=D_{\mathbf{f}}(\mathbf{u}_b; p_h)$ by solving
the local problem (\ref{W0}). This task can also be implemented in
parallel, and the computational complexity is proportional to the
number of unknowns.
\end{enumerate}

\end{algorithm3}

Note that, Step (2) in the Variable Reduction Algorithm 1 is the
only computation extensive part of the implementation.

\section{Numerical Experiments}\label{Section:NumericalResults}

The goal of this section is to report some numerical results for the
hybridization weak Galerkin finite element method proposed and
analyzed in previous sections.

A Schur complement technique of the HWG method is utilized to
decrease the degree of freedom.  {\color{black}For example, if
$\Omega=[0,1]\times[0,1]$ and the uniform triangulation is used with
mesh size $\sqrt2/n$, the number of elements is $N_T=2n^2$ and the
number of edges is $N_E=3n^2+2n$. If $k=1$, then the degree of
freedom for usual weak Galerkin method is $7N_T+2N_E=20n^2+4n$, the
degree of freedom for hybridized weak Galerkin method is
$13N_T+2N_E=32n^2+4n$, while the degree of freedom can be reduced to
$2N_E+N_T=8n^2+4n$ by using the Schur complement.}

Let $(\mathbf{u};p)$ be the exact solution of
(\ref{Stokes})-(\ref{boundarycond}) and $(\mathbf{u}_h;p_h)$ be the
numerical solution of (\ref{WG1})-(\ref{WG2}). Denote
$\mathbf{e}_h=Q_h\mathbf{u}-\mathbf{u}_h$ and
$\varepsilon_h=\widetilde{Q}_hp-p_h.$ The error for the weak
Galerkin solution is measured in four norms defined as follows:
\begin{eqnarray*}
\displaystyle|\!|\!| \mathbf{e}_h
|\!|\!|^2&=&\sum_{T\in\mathcal{T}_h} \left(\int_T|\nabla_w
\mathbf{e}_h|^2 dT+ h_T^{-1}\int_{\partial
T}(\mathbf{e}_0-\mathbf{e}_b)^2 ds\right), \nonumber
\\[1mm]
\displaystyle\|\mathbf{e}_h\|^2&=&\sum_{T\in\mathcal{T}_h}
\int_T|\mathbf{e}_h|^2 dT,
\\[1mm]
\displaystyle\|\varepsilon_h\|^2&=&\sum_{T\in\mathcal{T}_h}
\int_T|\varepsilon_h|^2 dT,
\\
\displaystyle\|Q_b\boldsymbol{\lambda}-\boldsymbol{\lambda}_h\|^2&=&\sum_{e\in\mathcal{E}_h}
h_e \int_e |Q_b\boldsymbol{\lambda}-\boldsymbol{\lambda}_h|^2 ds.
\\[1mm]
\end{eqnarray*}

\textbf{Example 7.1} {\rm Consider the problem
(\ref{Stokes})-(\ref{boundarycond}) in the square domain
$\Omega=(0,1)^2$. The HWG finite element space $k = 1$ is employed
in the numerical discretization. It has the analytic solution
\begin{eqnarray*}
\mathbf{u}=\left(\begin{array}{c} \sin(2\pi x) \cos(2\pi y)
\\
-\cos(2\pi x) \sin(2\pi y)
\end{array}\right) \ {\rm and\ } p=x^2 y^2-\frac19.
\end{eqnarray*}
The right hand side function $\mathbf{f}$ in (\ref{Stokes}) is
computed to match the exact solution. The mesh size is denoted by
$h$.

Table 7.1 shows that the errors and convergence rates of Example 7.1
in $|\!|\!| \cdot |\!|\!|-$ norm and $L^2-$norm for the HWG-FEM
solution $\mathbf{u}$ are of order $O(h)$ and $O(h^2)$ when $k=1$,
respectively.

Table 7.2 shows that the errors and orders of Example 7.1 in
$L^2-$norm for pressure and $\boldsymbol{\lambda}$ . The numerical
results are also consistent with theory for these two cases. }

\begin{center}
Table 7.1. Numerical errors and orders for $\mathbf{u}$ of Example 7.1. \\
\begin{center}
    \begin{tabular}{ | c || c | c | c | c |}
    \hline
     $h$  & \  $|\!|\!| \mathbf{e}_h|\!|\!|$\  & \quad order \quad \,  &  \  $\|\mathbf{e}_h\|$  \   & \ \  order \ \  \\ \hline \hline
     1/4  & \  5.8950e+00  \   &          & \ 1.3555e+00 \ &           \\ \hline
     1/8  &  2.9253e+00    & 1.0109   & 2.3750e-01   & 2.5128     \\ \hline
    1/16  &  1.4552e+00    & 1.0074   & 4.9049e-02   & 2.2756     \\ \hline
    1/32  &  7.2651e-01    & 1.0022   & 1.1500e-02   & 2.0926     \\ \hline
    1/64  &  3.6312e-01    & 1.0006   & 2.8254e-03   & 2.0251     \\ \hline
    1/128 &  1.8154e-01    & 1.0001   & 7.0325e-04   & 2.0063     \\ \hline
    \hline
    \end{tabular}
\end{center}
\end{center}

\begin{center}
Table 7.2. Numerical errors and orders for $p$ and $\boldsymbol{\lambda}$ of Example 7.1. \\
\begin{center}
    \begin{tabular}{ | c || c | c | c | c |}
    \hline
     $h$  & \  $\|\varepsilon_h\|$\  & \quad order \quad \,  &  \  $\|Q_b\boldsymbol{\lambda}-\boldsymbol{\lambda}_h\|$  \   & \ \  order \ \  \\ \hline \hline
     1/4  & \  5.1609e-01  \   &          & \ 8.6908e-01 \ &           \\ \hline
     1/8  &    2.9426e-01  & 0.8105   & 3.2951e-01   & 1.3992     \\ \hline
    1/16  &    1.4706e-01  & 1.0007   & 9.8958e-02   & 1.7354     \\ \hline
    1/32  &    7.2990e-02  & 1.0107   & 2.6698e-02   & 1.8901     \\ \hline
    1/64  &    3.6391e-02  & 1.0041   & 6.9159e-03   & 1.9487     \\ \hline
    1/128 &    1.8180e-02  & 1.0012   & 1.7681e-03   & 1.9677     \\ \hline
    \hline
    \end{tabular}
\end{center}
\end{center}

\textbf{Example 7.2} {\rm Consider the problem
(\ref{Stokes})-(\ref{boundarycond}) in the square domain
$\Omega=(0,1)^2$. The HWG finite element space $k = 1$ is employed
in the numerical discretization. It has the analytic solution
\begin{eqnarray*}
\mathbf{u}=\left(\begin{array}{c} -2xy(x-1)(y-1)x(x-1)(2y-1)
\\
2xy(x-1)(y-1)y(y-1)(2x-1)
\end{array}\right)
\end{eqnarray*}
and
$$
p=x^4+y^4-\frac25.
$$
The right hand side function $\mathbf{f}$ in (\ref{Stokes}) is
computed to match the exact solution. The mesh size is denoted by
$h$.

The numerical results are presented in Tables 7.3-7.4 which confirm
the theory developed in previous sections. }

\begin{center}
Table 7.3. Numerical errors and orders for $\mathbf{u}$ of Example 7.2. \\
\begin{center}
    \begin{tabular}{ | c || c | c | c | c |}
    \hline
     $h$  & \  $|\!|\!| \mathbf{e}_h|\!|\!|$\  & \quad order \quad \,  &  \  $\|\mathbf{e}_h\|$  \   & \ \  order \ \  \\ \hline \hline
     1/4  & \  2.8805e-01  \   &          & \ 4.2555e-02 \ &           \\ \hline
     1/8  &  1.4913e-01    & 0.9498   & 1.0184e-02   & 2.0630     \\ \hline
    1/16  &  7.5883e-02    & 0.9747   & 2.5894e-03   & 1.9756     \\ \hline
    1/32  &  3.8233e-02    & 0.9890   & 6.5809e-04   & 1.9762     \\ \hline
    1/64  &  1.9173e-02    & 0.9957   & 1.6598e-04   & 1.9872     \\ \hline
    1/128 &  9.5963e-03    & 0.9985   & 4.1663e-05   & 1.9942     \\ \hline
    \hline
    \end{tabular}
\end{center}
\end{center}

\begin{center}
Table 7.4. Numerical errors and orders for $p$ and $\boldsymbol{\lambda}$ of Example 7.2. \\
\begin{center}
    \begin{tabular}{ | c || c | c | c | c |}
    \hline
     $h$  & \  $\|\varepsilon_h\|$\  & \quad order \quad \,  &  \  $\|Q_b\boldsymbol{\lambda}-\boldsymbol{\lambda}_h\|$  \   & \ \  order \ \  \\ \hline \hline
     1/4  & \  7.7802e-02  \   &          & \ 1.8028e-01 \ &           \\ \hline
     1/8  &    3.7184e-02  & 1.0651   & 8.2876e-02   & 1.1212     \\ \hline
    1/16  &    1.4725e-02  & 1.3364   & 3.2275e-02   & 1.3605     \\ \hline
    1/32  &    5.1629e-03  & 1.5121   & 1.1179e-02   & 1.5297     \\ \hline
    1/64  &    1.6890e-03  & 1.6120   & 3.5619e-03   & 1.6500     \\ \hline
    1/128 &    5.4636e-04  & 1.6282   & 1.0720e-03   & 1.7324     \\ \hline
    \hline
    \end{tabular}
\end{center}
\end{center}

\Acknowledgements{This work was supported by National Natural
Science Foundation of China (Grant No. 11371171), and by the Program
for New Century Excellent Talents in University of Ministry of
Education of China.}


\end{document}